\documentclass[a4paper]{amsart}%
\usepackage{amscd}
\usepackage{amsthm}
\usepackage{graphicx}
\usepackage{amsmath}
\usepackage{amsfonts}
\usepackage{amssymb}%
\providecommand{\U}[1]{\protect\rule{.1in}{.1in}}

\newtheorem{theorem}{Theorem}

\newtheorem{corollary}[theorem]{Corollary}

\newtheorem{definition}[theorem]{Definition}
\newtheorem{example}[theorem]{Example}

\newtheorem{lemma}[theorem]{Lemma}

\newtheorem{remark}[theorem]{Remark}

\newcommand\supp{\mathop{\rm supp}}

\begin{document}

\title[Weighted estimates for generalized Riesz potentials]{Weighted estimates for generalized Riesz potentials}
\author{Pablo Rocha}
\address{Departamento de Matem\'atica, Universidad Nacional del Sur, Bah\'{\i}a Blanca, 8000 Buenos Aires, Argentina.}
\email{pablo.rocha@uns.edu.ar}
\thanks{\textbf{2.010 Math. Subject Classification}: 42B25, 42B30.}

\begin{abstract}
For $0 \leq \alpha < n$, $0 < p \leq 1$ and $\frac{1}{q} = \frac{1}{p} - \frac{\alpha}{n}$, we obtain the boundedness from $H^{p}_{w}(\mathbb{R}^{n})$ into $L^{q}_{w^{q/p}}(\mathbb{R}^{n})$ of certain generalized Riesz potentials for certain weights $w$. This result is achieved via the infinite atomic decomposition developed in \cite{Rocha}.
\end{abstract}

\maketitle

\section{Introduction}

Given $0 \leq \alpha < n$ and $m \in \mathbb{N} \cap \left(1 - \frac{\alpha}{n}, +\infty \right)$, we consider the operator $T_{\alpha, m}$ defined by
\begin{equation} \label{Talfa}
T_{\alpha, m}f(x) = \int_{\mathbb{R}^{n}} |x - A_1 y|^{-\alpha_1} \cdot \cdot \cdot |x - A_m y|^{-\alpha_m} f(y) \, dy, \,\,\,\,\,\,\,\,\, 
(x \in \mathbb{R}^{n}),
\end{equation}
where the $\alpha_j$'s are positive constants such that $\alpha_1 + \cdot \cdot \cdot + \alpha_m = n - \alpha$ and the $A_j$'s are 
$n \times n$ invertible matrices. For the case $\alpha = 0$, we require that the matrices $A_j$ satisfy the following additional condition: 
$A_i - A_j$ are invertible for $i \neq j$, $1 \leq i, j \leq m$. 

The germ of this operator appears in \cite{Ricci}, there F. Ricci and P. Sj\"ogren obtained the boundedness on $L^{p}(\mathbb{H}_1)$, $1 < p \leq +\infty$, for a family of maximal operators on the three dimensional Heisenberg group
$\mathbb{H}_1$. To get this result, they studied the $L^{2}(\mathbb{R})$ boundedness of the operator
\[
Tf(x) = \int_{\mathbb{R}} |x-y|^{\alpha-1}|(\beta-1)x- \beta y|^{-\alpha} f(y) \, dy,
\] 
for $\beta \neq 0, 1$ and $0 < \alpha < 1$. A generalization of this operator on $\mathbb{R}^{n}$ was studied by T. Godoy and M. Urciuolo in 
\cite{Godoy}.

If $0 < \alpha < n$ and $m \geq 1$, then the operator $T_{\alpha, m}$ has the same behavior that the Riesz potential on 
$L^{p}(\mathbb{R}^{n})$. Indeed
\begin{equation} \label{Talfa m}
|T_{\alpha, m}f(x) | \leq C \sum_{j=1}^{m} \int_{\mathbb{R}^{n}} |A_{j}^{-1}x - y |^{\alpha - n} |f(y)| dy = 
C \sum_{j=1}^{m} I_{\alpha}(|f|)(A_{j}^{-1}x),
\end{equation}
for all $x \in \mathbb{R}^{n}$, this pointwise inequality implies that $T_{\alpha, m}$ is a bounded operator from 
$L^{p}(\mathbb{R}^{n})$ into $L^{q}(\mathbb{R}^{n})$ for $1 < p < \frac{n}{\alpha}$ and $\frac{1}{q} = \frac{1}{p} - \frac{\alpha}{n}$, and it is of type weak $(1, n/n - \alpha)$. We observe that if $A_j = id$ for all $j=1, ..., m$, then $T_{\alpha, m}$ is the Riesz potential 
$I_{\alpha}$. Thus, for $0 < \alpha < n$ and $m \geq 1$, the operator $T_{\alpha, m}$ is a kind of generalization of the Riesz potential. 

It is well known that the Riesz potential $I_{\alpha}$ is bounded from $H^{p}(\mathbb{R}^{n})$ into $H^{q}(\mathbb{R}^{n})$ for $0 < p \leq 1$ and $\frac{1}{q} = \frac{1}{p} - \frac{\alpha}{n}$ (see \cite{Taibleson}, \cite{Krantz}). In \cite{R-U2}, the author jointly with M. Urciuolo proved the $H^{p}(\mathbb{R}^{n}) - L^{q}(\mathbb{R}^{n})$ boundedness of the operator $T_{\alpha, m}$ and we also showed that the $H^{p}(\mathbb{R}) - H^{q}(\mathbb{R})$ boundedness does not hold for 
$0 < p \leq \frac{1}{1+\alpha}$, $\frac{1}{q} = \frac{1}{p} - \alpha$ and $T_{\alpha, m}$ with $0 \leq \alpha <1$, $m=2$, $A_1 = 1$, and $A_2 =-1$. This is a significant difference with respect to the case $0 < \alpha < 1$, $n=m=1$, and $A_1=1$.

For the case $\alpha =0$ and $m \geq 2$, the operator $T_{0, m}$ was studied under the assumption that $A_i - A_j$ are invertible if $i \neq j$. 

The behavior of this class of operators and generalizations of them on the spaces of functions 
$L^{p}(\mathbb{R}^{n})$, $L^{p}_{w}(\mathbb{R}^{n})$, $H^{p}(\mathbb{R}^{n})$ was studied in \cite{Urciuolo}, \cite{Riveros}, \cite{R-U}, \cite{R-U2}, and \cite{ro2}. Their behavior on variable Hardy spaces $H^{p(\cdot)}(\mathbb{R}^{n})$ and on variable Hardy-Morrey spaces 
$\mathcal{HM}_{p(\cdot), u}(\mathbb{R}^{n})$ can be found in \cite{rochur}, \cite{R-U3}, \cite{ro3} and \cite{Tan}.

To obtain the boundedness of operators, like singular integrals or fractional type operators, in weighted Hardy spaces $H^{p}_{w}(\mathbb{R}^{n})$, one can appeal to the atomic or molecular characterization of $H^{p}_{w}(\mathbb{R}^{n})$, which means that a distribution in $H^{p}_w$ can be represented as a sum of atoms or molecules; see \cite{cuerva}, \cite{tor}, \cite{li}, \cite{lee}, and \cite{Rocha}. Then the boundedness of linear operators in $H^{p}_w$ can be deduced, in principle, from their behavior on atoms. However, it must be mentioned that in the classical Hardy spaces, i.e.: $w \equiv 1$, M. Bownik in \cite{bownik} has shown, by means of an example of Y. Meyer, that it not suffice to check that an operator from a Hardy spaces $H^{p}$, $0 < p \leq 1$, into some quasi Banach space $X$, maps atoms into bounded elements of $X$ to establish that this operator extends to  a bounded operator on $H^{p}$. This example is, in a certain sense, pathological. Now, if our operator $T$ under study is bounded on $L^{s}$ for some $1 < s < +\infty$, then the uniform boundedness of $T$ on atoms in $L^{p}$ norm (or $H^{p}$ norm) implies that the operator extends to a bounded operator from $H^{p}$ into $L^{p}$ 
(or on $H^{p}$), this follows from the fact that one always can take an atomic decomposition which converges in the norm of $L^{s}$, see \cite{dachun} and \cite{ro}. In \cite{Rocha} the author proved that if $w \in \mathcal{A}_{\infty}$, $0 < p \leq 1$ and 
$f \in \widehat{\mathcal{D}}_{0}$, then there exist a sequence of $w - (p, p_0, d)$ atoms $\{ a_j \}$ and a sequence of scalars $\{ \lambda_j \}$ with $\sum_{j} |\lambda_j |^{p} \leq c \| f \|_{H^{p}_{w}}^{p}$ such that $f = \sum_{j} \lambda_j a_j$, where the convergence is in $L^{s}(\mathbb{R}^{n})$ for each $1 < s <\infty$ (see Theorem \ref{decomp atomic} below). This result allows to avoid any problem that could arise with respect to establish the boundedness of classical operators on $H^{p}_{w}(\mathbb{R}^{n})$ (see \cite{Rocha}, Theorems 2.10 and 3.4). 

In this work, for certain weights $w$ and by means of the atomic decomposition developed in \cite{Rocha}, we will prove $H^{p}_{w}(\mathbb{R}^{n}) - L^{q}_{w^{q/p}}(\mathbb{R}^{n})$ estimates for the operator $T_{\alpha, m}$ given by (\ref{Talfa}), where $0 < p \leq 1$, $\frac{1}{q} = \frac{1}{p} - \frac{\alpha}{n}$.

To get our main results we will also consider the following condition for the weights: there exists $C > 0$ such that
\begin{eqnarray} \label{cond A}
w(A_j x) \leq C w(x), \,\,\,\, \text{a.e.} \, x \in \mathbb{R}^{n}, \, j=1, ..., m.
\end{eqnarray}

Next we give two nontrivial examples of weights that satisfy the condition (\ref{cond A}). Observe that also power weights satisfy this condition.
\begin{example} Let $w(x) = \left\{\begin{array}{cc}
                                                     \log(\frac{1}{|x|}) & |x| < \frac{1}{e}  \\
                                                     1 & |x| \geq \frac{1}{e}
                                                   \end{array}
                                                 \right.$. Then $w \in \mathcal{A}_1$ and satisfies (\ref{cond A}).
\end{example}
\begin{example}
The nonnegative function $w(x)=|x|^{a}$, $x \in \mathbb{R}^{n}\setminus\{0\}$, is a weight in the Muckenhoupt class $A_{p}$, 
$1 < p < \infty$, if and only if $-n < a < n (p-1)$ $($see p. 506 in \cite{grafakos}$)$. Such a weight satisfies (\ref{cond A}).
\end{example}

In Section 2 we present the basics of weighted theory and briefly recall the atomic decomposition developed in \cite{Rocha}. In Section 3 appears our main results. We study the cases $\alpha = 0$ and $0 < \alpha < n$ separately. More precisely, we obtain:

\

{\bf Case $\alpha = 0$.} \textit{If $w \in \mathcal{A}_{\infty}$ and satisfies (\ref{cond A}), then $T_{0, m}$ can be extended to an $H^{p}_{w}(\mathbb{R}^{n})$-$L^{p}_{w}(\mathbb{R}^{n})$ bounded operator for each $0 < p \leq 1$} (see Theorem \ref{Thm 1} below).

\

{\bf Case $0 < \alpha < n$.} \textit{If $w^{\frac{n}{(n-\alpha)s}} \in \mathcal{A}_{1}$ with $0 < s < 1$, 
$\frac{r_w}{r_w - 1} < \frac{n}{\alpha}$ and $w$ satisfies (\ref{cond A}), then $T_{\alpha, m}$ can be extended to an $H^{p}_{w^{p}}(\mathbb{R}^{n})$-$L^{q}_{w^{q}}(\mathbb{R}^{n})$ bounded operator for each $s \leq p \leq 1$ and 
$\frac{1}{q} = \frac{1}{p} - \frac{\alpha}{n}$} (see Theorem \ref{Ta} below).

\

\textbf{Notation:} The symbol $A\lesssim B$ stands for the inequality $A\leq cB$ for some constant $c$. We denote by $B(x_0, r)$ the ball centered at $x_0 \in \mathbb{R}^{n}$ of radius $r$. For a measurable subset $E \subset \mathbb{R}^{n}$ we denote $|E|$ and $\chi_E$ the Lebesgue measure of $E$ and the characteristic function of $E$ respectively. Given a real number $s \geq 0$, we write $\lfloor s \rfloor$ for the integer part of $s$. As usual we denote with $\mathcal{S}(\mathbb{R}^{n})$ the space of smooth and rapidly decreasing functions, with $\mathcal{S}'(\mathbb{R}^{n})$  the dual space. If $\beta$ is the multiindex $\beta=(\beta_1, ..., \beta_n)$,
then $|\beta| = \beta_1 + ... + \beta_n$. 

Throughout this paper, $C$ will denote a positive constant, not necessarily the same at each occurrence.

\section{Preliminaries} 

A weight is a non-negative locally integrable function on $\mathbb{R}^{n}$ that takes values in $(0, \infty)$ almost everywhere, i.e. : the weights are allowed
to be zero or infinity only on a set of Lebesgue measure zero.

Given a weight $w$ and $0 < p < \infty$, we denote by $L^{p}_{w}(\mathbb{R}^{n})$ the spaces of all functions $f$ satisfying $\| f \|_{L^{p}_{w}}^{p} := \int_{\mathbb{R}^{n}} |f(x)|^{p} w(x) dx < \infty$ . When $p=\infty$, we have that $L^{\infty}_{w}(\mathbb{R}^{n}) =L^{\infty}(\mathbb{R}^{n})$ with $\| f \|_{L^{\infty}_{w}} = \| f \|_{L^{\infty}}$. If $E$ is a measurable set, we use the notation $w(E) = \int_{E} w(x) dx$.

Let $f$ be a locally integrable function on $\mathbb{R}^{n}$. The function
$$M(f)(x) = \sup_{B \ni x} \frac{1}{|B|} \int_{B} |f(y)| dy,$$
where the supremum is taken over all balls $B$ containing $x$, is called the uncentered Hardy-Littlewood maximal function of $f$.

We say that a weight $w \in \mathcal{A}_1$ if there exists $C >  0$ such that
\begin{equation*}
M(w)(x) \leq C w(x), \,\,\,\,\, a.e. \, x \in \mathbb{R}^{n}, 
\end{equation*}
the best possible constant is denoted by $[w]_{\mathcal{A}_1}$. Equivalently, a weight $w \in \mathcal{A}_1$ if there exists $C >  0$ such that for every ball $B$
\begin{equation}
\frac{1}{|B|} \int_{B} w(x) dx \leq C \, ess\inf_{x \in B} w(x). \label{A1condequiv}
\end{equation}
\begin{remark} If $w \in \mathcal{A}_1$ and $0 < r < 1$, then by H\"older inequality we have that $w^{r} \in \mathcal{A}_1$.
\end{remark}
For $1 < p < \infty$, we say that a weight $w \in \mathcal{A}_p$ if there exists $C> 0$ such that for every ball $B$
$$\left( \frac{1}{|B|} \int_{B} w(x) dx \right) \left( \frac{1}{|B|} \int_{B} [w(x)]^{-\frac{1}{p-1}} dx \right)^{p-1} \leq C.$$
It is well known that $\mathcal{A}_{p_1} \subset \mathcal{A}_{p_2}$ for all $1 \leq p_1 < p_2 < \infty$. Also, if $w \in \mathcal{A}_{p}$ with $1 < p < \infty$, then there exists $1 < q <p$ such that $w \in \mathcal{A}_{q}$. We denote by $\widetilde{q}_{w} = \inf \{ q>1 : w \in \mathcal{A}_q \}$ \textit{the critical index of $w$.}

\begin{lemma} \label{prop doblante} $($Proposition 7.1.5 in \cite{grafakos}$)$ If $w \in \mathcal{A}_{p}$ for some $1 \leq p < \infty$, then the measure $w(x) dx$ is doubling: precisely, for all $\lambda >1$ an all ball $B$ we have
\[
w(\lambda B) \leq \lambda^{np} [w]_{\mathcal{A}_p} w(B),
\]
where $\lambda B$ denotes the ball with the same center as $B$ and radius $\lambda$ times the radius of $B$.
\end{lemma}

\begin{lemma} \label{prop A} If $w \in \mathcal{A}_{\infty}$ and $A$ is a $n \times n$ invertible matrix such that $w(Ax) \lesssim w(x)$ a.e. $x \in \mathbb{R}^{n}$, then for each $M>0$ there exists a constant $C > 0$ such that for every ball $B(x_0, r)$ holds
\[
w(B(Ax_0, 2Mr)) \leq C w(B(x_0, r)).
\]
\end{lemma}

\begin{proof} Applying a change of variable and Lemma \ref{prop doblante}, the lemma follows.
\end{proof}

\begin{theorem} $($\textit{Theorem 9 in \cite{Muck}}$)$ \label{cond Ap} Let $1 < p < \infty$. Then
\[
\int_{\mathbb{R}^{n}} [Mf(x)]^{p} w(x) dx \leq C_{w, p, n} \int_{\mathbb{R}^{n}} |f(x)|^{p} w(x) dx,
\]
for all $f \in L^{p}_{w}(\mathbb{R}^{n})$ if and only if $w \in \mathcal{A}_p$. 
\end{theorem}

Given $1 < p \leq q < \infty$, we say that a weight $w \in \mathcal{A}_{p,q}$ if there exists $C> 0$ such that for every ball $B$
$$\left( \frac{1}{|B|} \int_{B} [w(x)]^{q} dx \right)^{1/q} \left( \frac{1}{|B|} \int_{B} [w(x)]^{-p'} dx \right)^{1/p'} \leq C < \infty.$$
For $p=1$, we say that a weight $w \in \mathcal{A}_{1,q}$ if there exists $C> 0$ such that for every ball $B$
$$\left( \frac{1}{|B|} \int_{B} [w(x)]^{q} dx \right)^{1/q} \leq C \, ess\inf_{x \in B} w(x).$$
When $p=q$, this definition is equivalent to $w^{p} \in \mathcal{A}_{p}$.
\begin{remark} \label{w_1_q}
From the inequality in (\ref{A1condequiv}) it follows that if a weight $w \in \mathcal{A}_1$, then $0 < ess\inf_{x \in B} w(x) < \infty$ for each ball $B$. Thus $w \in \mathcal{A}_1$ implies that $w^{\frac{1}{q}} \in \mathcal{A}_{p,q}$, for each $1 \leq p \leq q < \infty$.
\end{remark}
Given $0 < \alpha < n$, we define the fractional maximal operator $M_{\alpha}$ by
\[
M_{\alpha}f(x) = \sup_{B \ni x} \frac{1}{|B|^{1 - \frac{\alpha}{n}}} \int_{B} |f(y) | dy,
\]
where $f$ is a locally integrable function and the supremum is taken over all balls $B$ containing $x$

The fractional maximal operators satisfies the following weighted inequality.

\begin{theorem}$($\textit{Theorem 3 in \cite{Muck2}}$)$ \label{Teo A_pq} If $0 < \alpha < n$, $1 < p < \frac{n}{\alpha}$, $\frac{1}{q} = \frac{1}{p} - \frac{\alpha}{n}$ and $w \in \mathcal{A}_{p,q}$, then
\[
\left(\int_{\mathbb{R}^{n}} [M_{\alpha}f(x)]^{q} w^{q}(x) dx\right)^{1/q} \leq C \left( \int_{\mathbb{R}^{n}} |f(x)|^{p} w^{p}(x) dx \right)^{1/p},
\]
for all $f \in L^{p}_{w^{p}}(\mathbb{R}^{n})$ . 
\end{theorem}

A weight $w$ satisfies the reverse H\"older inequality with exponent $s > 1$, denoted by $w \in RH_{s}$, if there exists $C> 0$ such that for every ball $B$,
$$\left(\frac{1}{|B|} \int_{B} [w(x)]^{s} dx \right)^{\frac{1}{s}} \leq C \frac{1}{|B|} \int_{B} w(x) dx;$$
the best possible constant is denoted by $[w]_{RH_s}$. We observe that if $w \in RH_s$, then by H\"older's inequality, $w \in RH_t$ for all $1 < t < s$, and
$[w]_{RH_t} \leq [w]_{RH_s}$. Moreover, if $w \in RH_{s}$, $s >1$, then $w \in RH_{s+ \epsilon}$ for some $\epsilon >0$. We denote by $r_w = \sup \{ r >1 : w \in RH_{r} \}$ \textit{the critical index of $w$ for the reverse H\"older condition.}

It is well known that a weight $w$ satisfies the condition $\mathcal{A}_{\infty}$ if and only if $w \in \mathcal{A}_p$ for some $p \geq 1$ (see corollary 7.3.4 in \cite{grafakos}). So $\mathcal{A}_{\infty} = \cup_{1 \leq p < \infty} \mathcal{A}_p$. Also, $w \in \mathcal{A}_{\infty}$ if and only if $w \in RH_{s}$ for some $s>1$ (see Theorem 7.3.3 in \cite{grafakos}). Thus $1 < r_w \leq +\infty$ for all $w \in \mathcal{A}_{\infty}$.

Other remarkable result about the reverse H\"older classes was discovered by Stromberg and Wheeden, they proved in \cite{wheeden} that $w \in RH_{s}$, $1 < s < + \infty$, if and only if  $w^{s} \in \mathcal{A}_{\infty}$.

Given a weight $w$, $0 < p < \infty$ and a measurable set $E$ we set $w^{p}(E) = \int_{E} [w(x)]^{p} dx$. The following result is a immediate consequence of the reverse H\"older condition.

\begin{lemma} \label{ineq RH} For $0 < \alpha < n$, let $0 < p < \frac{n}{\alpha}$ and $\frac{1}{q} = \frac{1}{p} - \frac{\alpha}{n}$. 
If $w^{p} \in RH_{\frac{q}{p}}$, then
\[
[w^{p}(B)]^{-\frac{1}{p}} [w^{q}(B)]^{\frac{1}{q}} \leq [w^{p}]_{RH_{q/p}}^{1/p} |B|^{-\frac{\alpha}{n}},
\]
for each ball $B$ in $\mathbb{R}^{n}$.
\end{lemma}

\begin{lemma} $($Lemma 4.6 in \cite{Rocha}$)$ \label{cond r_w}
Let  $0 < p < 1$. If $w^{1/p} \in \mathcal{A}_{1}$, then $p \cdot r_{w^{p}} \leq r_{w} \leq r_{w^{p}}.$
\end{lemma}

\begin{lemma} $($Lemma 4.7 in \cite{Rocha}$)$ \label{cond r_w_p_q}
Let  $0 < p < q$. If $w^{q} \in \mathcal{A}_{1}$, then $ p \cdot r_{w^{p}} \leq q \cdot r_{w^{q}}$.
\end{lemma}

We conclude these preliminaries with weighted Hardy theory. Topologize $\mathcal{S}(\mathbb{R}^{n})$ by the collection of semi-norms $\| \cdot \|_{\alpha, \beta}$, with $\alpha$ and $\beta$ multi-indices, given by
$$\| \varphi \|_{\alpha, \beta} = \sup_{x \in \mathbb{R}^{n}} |x^{\alpha} \partial^{\beta}\varphi(x)|.$$
For each $N \in \mathbb{N}$, we set $\mathcal{S}_{N}=\left\{  \varphi\in \mathcal{S}(\mathbb{R}^{n}): \| \varphi \|_{\alpha, \beta} \leq 1, |\alpha|, |\beta| \leq N \right\}$. Let $f \in \mathcal{S}'(\mathbb{R}^{n})$, we denote by $\mathcal{M}_{N}$ the grand maximal
operator given by
\[
\mathcal{M}_{N}f(x)=\sup\limits_{t>0}\sup\limits_{\varphi\in\mathcal{S}_{N}
}\left\vert \left(  t^{-n}\varphi(t^{-1} \cdot)\ast f\right)  \left(  x\right)
\right\vert.
\]
Given a weight $w \in \mathcal{A}_{\infty}$ and $p > 0$, the weighted Hardy space $H^{p}_w(\mathbb{R}^{n})$ consists of all tempered distributions $f$ such that
\[
\| f \|_{H^{p}_w(\mathbb{R}^{n})} = \| \mathcal{M}_{N}f \|_{L^{p}_w(\mathbb{R}^{n})} = \left( \int_{\mathbb{R}^{n}}  [\mathcal{M}_{N}f(x)]^{p} w(x) dx \right)^{1/p} < \infty.
\]

Let $\phi \in \mathcal{S}(\mathbb{R}^{n})$ be a function such that $\int \phi(x) dx \neq 0$. For $f \in \mathcal{S}'(\mathbb{R}^{n})$, we define the maximal function $M_{\phi}f$ by
$$M_{\phi}f(x)= \sup_{t>0} \left\vert \left(  t^{-n}\phi(t^{-1} \, \cdot)\ast f\right)  \left(  x\right)\right\vert.$$
For $N$ sufficiently large, we have  $\| M_{\phi}f \|_{L^{p}_w} \simeq \| \mathcal{M}_{N}f \|_{L^{p}_w}$, (see \cite{tor}). The following set
\[
\widehat{\mathcal{D}}_{0} = \{ \phi \in \mathcal{S}(\mathbb{R}^{n}) : \widehat{\phi} \in C_{c}^{\infty}(\mathbb{R}^{n}), \textit{and} \, \supp ( \hat{\phi} ) \subset \mathbb{R}^{n} \setminus B(0, \delta) \,\, \textit{for some} \,\, \delta >0 \}
\]
is dense in $H^{p}_{w}(\mathbb{R}^{n})$, $0 < p  < \infty$, if $w$ is a doubling weight on $\mathbb{R}^{n}$ (see \cite{tor}, Theorem 1, pp. $103$).

We adopt the weighted atomic decomposition developed in \cite{Rocha}. Let $w \in \mathcal{A}_{\infty}$ with critical index $\widetilde{q}_w$ and critical index $r_w$ for reverse H\"older condition. Let $0 < p \leq 1$, $\max \{ 1, p(\frac{r_w}{r_w -1}) \} < p_0 \leq +\infty$,
and $d \in \mathbb{Z}$ such that $d \geq \lfloor n(\frac{\widetilde{q}_w}{p} -1) \rfloor$, a function $a(\cdot)$ is a $w - (p, p_0, d)$ atom centered in $x_0 \in \mathbb{R}^{n}$ if

\

$(a1)$ $a \in L^{p_0}(\mathbb{R}^{n})$ with support in the ball  $B= B(x_0, r)$.

\

$(a2)$ $\| a \|_{L^{p_0}(\mathbb{R}^{n})} \leq |B|^{\frac{1}{p_0}}w(B)^{-\frac{1}{p}}$.

\

$(a3)$ $\int x^{\beta} a(x) dx  =0$ for all multi-index $\beta$ such that $| \beta | \leq d$.

\begin{remark} \label{RH cond}
The condition $\max \{ 1, p(\frac{r_w}{r_w -1}) \} < p_0 < +\infty$ implies that $w \in RH_{(\frac{p_0}{p})'}$. If $r_w = +\infty$, then $w \in RH_{t}$ for each $1 < t < +\infty$. So, if $r_w = + \infty$, it considers $\frac{r_w}{r_w - 1} =1$. 
\end{remark}

The following theorem is crucial to get the main results.

\begin{theorem} $($Theorem 9 in \cite{Rocha}$)$ \label{decomp atomic} Let $f \in \hat{\mathcal{D}}_{0}$, and $0 < p \leq 1$. If $w \in \mathcal{A}_{\infty}$, then there exist a sequence of $w - (p, p_0, d)$ atoms $\{ a_j \}$ and a sequence of scalars $\{ \lambda_j \}$ with $\sum_{j} |\lambda_j |^{p} \leq c \| f \|_{H^{p}_{w}}^{p}$ such that $f = \sum_{j} \lambda_j a_j$, where the convergence is both in $L^{s}(\mathbb{R}^{n})$ and pointwise for each $1 < s <\infty$.
\end{theorem}

\section{Main results}

In the sequel, given a ball $B = B(x_0,r)$ and $A_1, ..., A_m$ $n \times n$ invertible matrices, we put 
$M= \max \{\| A_j \| : 1 \leq j \leq m \}$, $B_{i}^{\ast} = B(A_i x_0, 2Mr)$ for $i=1, ..., m$ and decompose $\mathbb{R}^{n} = \left( \bigcup_{i=1}^{m} B_{i}^{\ast} \right) \cup R$, where $R= \mathbb{R}^{n} \setminus \left( \bigcup_{i=1}^{m} B_{i}^{\ast} \right)$. Moreover,
$R = \bigcup_{k=1}^{m} R_k$, where
\[
R_{k} = \{ x \in R : |x - A_k x_0| \leq |x - A_i x_0| \,\, for \,\, all \,\, i \neq k \},
\]
for $k=1, ..., m$.

\begin{lemma} \label{estimT}For $0 \leq \alpha < n$ and $m \in \mathbb{N} \cap (1- \frac{\alpha}{n}, +\infty)$, let $T_{\alpha, m}$ be the operator \, defined by 
$(\ref{Talfa})$, where the $A_j$'s are $n \times n$ invertible matrices. If $w \in \mathcal{A}_{\infty}$ and $a(\cdot)$ is a $w-(p, p_0, d)$ atom supported on the ball $B=B(x_0, r)$, then
\[
|T_{\alpha, m}a(x)| \lesssim [w(B)]^{-1/p} \sum_{k=1}^{m} \chi_{R_k}(x) \left( M_{\frac{\alpha n}{n+d+1}}\left( \chi _{B}\right)
(A_{k}^{-1}x)\right) ^{\frac{n+d+1}{n}}, \,\,\,\,\, \text{if} \,\, x \in R.
\]
\end{lemma}

\begin{proof} We denote $k(x,y)=\left\vert x-A_{1}y\right\vert ^{-\alpha_{1}} \cdot \cdot \cdot 
\left\vert x-A_{m}y\right\vert ^{-\alpha _{m}}.$ In view of the moment condition of $a(.)$ we have
\[
T_{\alpha, m} a(x)=\int\limits_{B(x_0, r)}k(x,y)a(y)dy=\int\limits_{B(x_0, r)}\left( k(x,y)-q_{d}\left(
x,y\right) \right) a(y)dy, \,\,\,(x \in R),
\]
\newline
where $q_{d}$ is the degree $d$ Taylor polynomial of the function $
y\rightarrow k(x,y)$ expanded around $x_0$. By the standard estimate of the
remainder term of the Taylor expansion, there exists $\xi $ between $y$ and
$x_0$ such that
\[
\left\vert k(x,y)-q_{d}\left( x,y\right) \right\vert \lesssim \left\vert
y-x_0 \right\vert ^{d+1}\sum\limits_{k_{1}+...+k_{n}=d+1}\left\vert \frac{%
\partial ^{d+1}}{\partial y_{1}^{k_{1}}...\partial y_{n}^{k_{n}}}k(x,\xi
)\right\vert
\]
\[
\leq C \left\vert y-x_0 \right\vert ^{d+1}\left(
\prod\limits_{i=1}^{m}\left\vert x-A_{i}\xi \right\vert ^{-\alpha
_{i}}\right) \left( \sum\limits_{l=1}^{m}\left\vert x-A_{l}\xi \right\vert
^{-1}\right) ^{d+1}.
\]
Now, we decompose $R = \bigcup_{k=1}^{m} R_{k}$ where
$$R_{k} = \{ x \in R : |x - A_k x_0| \leq |x - A_i x_0| \,\, for \,\, all \,\, i \neq k \}.$$
If $x \in R$ then $|x - A_i x_0| \geq 2Mr$, since $\xi \in B$ it follows that $|A_i x_0 - A_i \xi | \leq  Mr \leq \frac{1}{2} |x - A_i x_0|$ so
$$|x - A_i \xi| = |x - A_i x_0 + A_i x_0 - A_i \xi| \geq |x - A_i x_0| - |A_i x_0 - A_i \xi| \geq \frac{1}{2} |x - A_i x_0|.$$
If $x \in R$, then $x \in R_{k}$ for some $k$ and since $\alpha_{1}+...+\alpha_{m} = n - \alpha$ we obtain
$$
\left\vert k(x,y)-q_{d}\left( x,y\right) \right\vert  \leq C
\left\vert y-x_0 \right\vert ^{d+1}\left( \prod\limits_{i=1}^{m}\left\vert
x-A_{i}x_0 \right\vert ^{-\alpha _{i}}\right) \left(
\sum\limits_{l=1}^{m}\left\vert x-A_{l}x_0 \right\vert ^{-1}\right) ^{d+1}
$$
$$
\leq C r^{d+1}\left\vert x-A_{k}x_0 \right\vert ^{-n+\alpha -d-1},
$$
this inequality allows us to conclude that
\begin{eqnarray*}
\left\vert T_{\alpha, m}a(x)\right\vert  \leq C\left\Vert a \right\Vert
_{1}r^{d+1}\left\vert x-A_{k}x_0\right\vert ^{-n+\alpha -d-1}\\
\leq C \left\vert B \right\vert ^{1-\frac{1}{p_0}}\left\Vert
a \right\Vert _{p_{0}}r^{d+1}\left\vert x-A_{k}x_0\right\vert ^{-n+\alpha -d-1},
\end{eqnarray*}
since $\|a \|_{p_0} \leq |B|^{1/p_0} [w(B)]^{-1/p}$, we have
\begin{eqnarray} \label{Tmalpha}
\left\vert T_{\alpha, m}a(x)\right\vert  \lesssim \frac{r^{n+d+1} \left\vert x-A_{k}x_0 \right\vert^{-n+\alpha -d-1}}{[w(B)]^{1/p}}
 \end{eqnarray}
\begin{eqnarray*}
\lesssim \frac{\left( M_{\frac{\alpha n}{n+d+1}}\left( \chi _{B}\right)(A_{k}^{-1}x)\right)^{\frac{n+d+1}{n}}}{[w(B)]^{1/p}}, \,\,\,\,
\text{if} \,\, x \in R_{k}.
\end{eqnarray*}
This completes the proof.
\end{proof}

\begin{remark} In Lemma \ref{estimT}, we observe that the matrices $A_j$ can be equal to each other.
\end{remark}

We recall the definition of the critical indices for a weight $w$.

\begin{definition} Given a weight $w$, we denote by $\widetilde{q}_{w} = \inf \{ q>1 : w \in \mathcal{A}_q \}$ \textit{the critical index of $w$}, and we denote by $r_{w} = \sup \{ r > 1 : w \in RH_{r} \}$ \textit{the critical index of $w$  for the reverse H\"older condition.}
\end{definition}

\begin{theorem} \label{Thm 1}
For $\alpha=0$ and $m \geq 2$, let $T_{0, m}$ be the operator defined by $(\ref{Talfa})$ where the $A_j$'s are $n\times n$ invertible matrices such that $A_i - A_j$ are invertible for $i \neq j$, $1 \leq i, j \leq m$. 
If $w \in \mathcal{A}_{\infty}$ and $w(A_j x) \lesssim w(x)$ a.e. $x \in \mathbb{R}^{n}$, $1 \leq j \leq m$, then $T_{0, m}$ can be extended to an $H^{p}_{w}(\mathbb{R}^{n})$-$L^{p}_{w}(\mathbb{R}^{n})$ bounded operator for each $0 < p \leq 1$.
\end{theorem}

\begin{proof} Theorem 6 in \cite{R-U3} asserts that the operator $T_{0, m}$ is bounded on $L^{p_0}(\mathbb{R}^{n})$ for each 
$1 < p_0 < +\infty$. Given $0 < p \leq 1$, we take $d= \lfloor n (\frac{\widetilde{q}_w}{p} -1 ) \rfloor$ and $p_0 \in (1, +\infty)$ such that $\frac{r_w}{r_w - 1} < p_0$. Thus, by Theorem 2.10 in \cite{Rocha}, it suffices to show that the operator $T_{0, m}$ is uniformly bounded in $L^{p}_{w}$ norm on all $w$-$(p, p_0, d)$ atoms $a(\cdot)$. 

Given a $w$-$(p, p_0, d)$ atom $a(\cdot)$ supported on the ball $B=B(x_0, r)$, let $B^{\ast}_{i} = B(A_i x_0, 2Mr)$. We write
\[
\int_{\mathbb{R}^{n}} |T_{0, m}a(x)|^{p} w(x) dx = \int_{\bigcup_{i=1}^{m} B_{i}^{\ast}} |T_{0, m}a(x)|^{p} w(x) dx 
\]
\[
+ \int_{\mathbb{R}^{n} \setminus (\bigcup_{i=1}^{m} B_{i}^{\ast})} |T_{0, m}a(x)|^{p} w(x) dx = I_1 + I_2
\]
To estimate the term $I_1$, we observe that $p \frac{r_w}{r_w - 1} < p_0$ thus $w \in RH_{(\frac{p_0}{p})'}$ (see Remark \ref{RH cond}), then H\"older's inequality  applied with $\frac{p_0}{p}$, the condition $(a2)$ of the atom $a(\cdot)$, and Lemma \ref{prop A} give
\[
I_1 \lesssim \| T_{0, m} a \|_{p_0}^{p} \sum_{i=1}^{m} \left( \int_{B_{i}^{\ast}} [w(x)]^{(p_0/p)'} \right)^{1/(p_0/p)'}
\]
\[
\lesssim \| a \|_{p_0}^{p} \sum_{i=1}^{m} |B_{i}^{\ast}|^{-p/p_0} w(B_{i}^{\ast}) \leq C.
\]
Taking account that $w(A_j x) \lesssim w(x)$ a.e. $x \in \mathbb{R}^{n}$, $1 \leq j \leq m$, and $w \in \mathcal{A}_{p \frac{n+d+1}{n}}$ ($p \frac{n+d+1}{n} > \widetilde{q}_w$), from Lemma \ref{estimT} with $\alpha =0$ and Theorem \ref{cond Ap}, we get
\[
I_2 \lesssim w(B)^{-1} \sum_{k=1}^{m} \int_{\mathbb{R}^{n}} [M(\chi_B)(x)]^{p\frac{n+d+1}{n}} w(A_k x) dx \leq C.
\]
Thus, $T_{0, m}$ is uniformly bounded in $L^{p}_{w}$ norm on every $w$-$(p, p_0, d)$ atom $a(\cdot)$. 
\end{proof}

\begin{theorem} \label{Ta}
For $0 < \alpha < n$ and $m \geq 1$, let $T_{\alpha, m}$ be the operator defined by $(\ref{Talfa})$ where the $A_j$'s are 
$n\times n$ invertible matrices. If $w^{\frac{n}{(n-\alpha)s}} \in \mathcal{A}_{1}$ with $0 < s < 1$, 
$\frac{r_w}{r_w - 1} < \frac{n}{\alpha}$ and $w(A_j x) \lesssim w(x)$ a.e. $x \in \mathbb{R}^{n}$, $1 \leq j \leq m$, then $T_{\alpha, m}$ can be extended to an $H^{p}_{w^{p}}(\mathbb{R}^{n})$-$L^{q}_{w^{q}}(\mathbb{R}^{n})$ bounded operator for each $s \leq p \leq 1$ and 
$\frac{1}{q} = \frac{1}{p} - \frac{\alpha}{n}$.
\end{theorem}

\begin{proof} For $0 < p \leq 1$ and $\frac{1}{q} = \frac{1}{p} - \frac{\alpha}{n}$, we have that $p < q \leq \frac{n}{n - \alpha}$. Now,
the condition $w^{n/(n - \alpha)s} \in \mathcal{A}_1$, $0 < s < 1 < \frac{n}{n - \alpha}$, implies that 
$w$, $w^{1/p}$, $w^{p}$ and $w^{q}$ belong to $\mathcal{A}_1$, for all $s \leq p \leq 1$ and $\frac{1}{q} = \frac{1}{p} - \frac{\alpha}{n}$.

We also have that $\widetilde{q}_{w^{p}} =1$, since $w^{p} \in \mathcal{A}_1$.
We choose $d= \lfloor n (\frac{1}{p} - 1) \rfloor$ and take $p_0$ such that $\frac{r_w}{r_w - 1} < p_0 < \frac{n}{\alpha}$, from Lemma \ref{cond r_w}, we have that 
$p \frac{r_{w^{p}}}{r_{w^{p}} -1} \leq \frac{r_w}{r_w -1} < p_0$ so $\max\{1, p (\frac{r_{w^{p}}}{r_{w^{p}} - 1}) \} < p_0$ (this last inequality is required in the definition of $w^{p}$-atom). By Theorem \ref{decomp atomic}, given $f \in \widehat{\mathcal{D}}_{0}$ we can write $f = \sum \lambda_j a_j$ where the $a_j$'s are 
$w^{p} - (p, p_0, d)$ atoms, the scalars $\lambda_j$ satisfies $\sum_{j} |\lambda_j |^{p} \lesssim \| f \|_{H^{p}_{w^{p}}}^{p}$  and the series converges in $L^{p_0}(\mathbb{R}^{n})$. For $\frac{1}{q_0} = \frac{1}{p_0} - \frac{\alpha}{n}$, $T_{\alpha, m}$ is a bounded operator from $L^{p_0}(\mathbb{R}^{n})$ into $L^{q_0}(\mathbb{R}^{n})$, this follows from the pointwise estimate that appears in (\ref{Talfa m}). Since $f = \sum_j \lambda_j a_j$ in $L^{p_0}(\mathbb{R}^{n})$, we have that
\begin{equation}
|T_{\alpha, m}f(x)| \leq \sum_{j} |\lambda_j| |T_{\alpha, m}a_j(x)|, \,\,\,\,\, \textit{a.e.} \, x \in \mathbb{R}^{n}. \label{puntual}
\end{equation}
If $\|T_{\alpha, m} a_j \|_{L^{q}_{w^{q}}} \leq C$, with $C$ independent of the $ w^{p} - (p, p_0, d)$ atom $a_j(\cdot)$, then (\ref{puntual}) allows us to obtain
\[
\|T_{\alpha, m} f \|_{L^{q}_{w^{q}}} \leq C \left( \sum_{j} |\lambda_j|^{\min\{1, q \}} \right)^{\frac{1}{\min\{1, q \}}} \leq C \left( \sum_{j} |\lambda_j |^{p} \right)^{1/p} \lesssim \| f \|_{H^{p}_{w^{p}}},
\] 
for all $f \in \widehat{\mathcal{D}}_0$, so the theorem follows from the density of  $\widehat{\mathcal{D}}_0$ in $H^{p}_{w^{p}}(\mathbb{R}^{n})$.

\

To conclude the proof we will prove that there exists an universal constant $C > 0$ such that 
\begin{equation} \label{uniform estimate}
\|T_{\alpha, m} a \|_{L^{q}_{w^{q}}} \leq C, \,\,\,\, \textit{for all} \,\, w^{p} - (p, p_0, d) \,\, \textit{atom} \,\, a(\cdot). 
\end{equation}
To prove (\ref{uniform estimate}), let $B_{i}^{\ast}=B(A_i x_0, 2Mr)$ for $i=1, ..., m$, where $B=B(x_0, r)$ is the ball containing the support of the atom $a(\cdot)$. So
\[
\int_{\mathbb{R}^{n}} |T_{\alpha, m}a(x)|^{q} w^{q}(x) dx = \int_{\bigcup_{i=1}^{m} B_{i}^{\ast}} |T_{\alpha, m}a(x)|^{q} w^{q}(x) dx
\]
\[  
+ \int_{\mathbb{R}^{n} \setminus (\bigcup_{i=1}^{m} B_{i}^{\ast})} |T_{\alpha, m}a(x)|^{q} w^{q}(x) dx = J_1 + J_2
\]
To estimate $J_1$, we apply H\"older's inequality with $\frac{q_0}{q}$ and use that $w^{q} \in RH_{(\frac{q_0}{q})'}$ ($p_0 > p \frac{r_{w^{p}}}{r_{w^{p}} - 1}$ and Lemma \ref{cond r_w_p_q} imply that $q_0 > q \frac{r_{w^{q}}}{r_{w^{q}}-1})$, thus
\[
J_1 \lesssim \|T_{\alpha, m} a \|_{L^{q_0}}^{q} \sum_{i=1}^{m} \left( \int_{B_{i}^{\ast}} [w^{q}(x)]^{(\frac{q_0}{q})'} dx \right)^{1/(\frac{q_0}{q})'}
\]
\[
\lesssim |B|^{q/p_0} (w^{p}(B))^{-q/p} \sum_{i=1}^{m}|B_{i}^{\ast}|^{-q/ q_0} w^{q}(B_{i}^{\ast})
\]
Lemma \ref{prop A} implies
\[
\lesssim |B|^{q\alpha/n}  (w^{p}(B))^{-q/p} w^{q}(B).
\]
Then, Lemma \ref{ineq RH} gives
\begin{equation}
J_1 \leq C. \label{Talfa 2}
\end{equation}

Now we estimate $J_2$. By Lemma \ref{estimT}, considering there $d= \lfloor n(\frac{1}{p}-1) \rfloor$ and $w^{p}$ instead of $w$, we have
\[
|T_{\alpha, m}a(x)|^{q} \lesssim (w^{p}(B))^{-q/p} \sum_{k=1}^{m} \left[ M_{\frac{\alpha n}{n+d+1}}(\chi_{B}) (A_{k}^{-1}x) \right]^{q\frac{n+d+1}{n}}, \,\, \textit{for all} \,\, x \notin \bigcup_{i=1}^{m} B_{i}^{\ast},
\]
by integrating this pointwise estimate over $\mathbb{R}^{n} \setminus (\bigcup_{i=1}^{m} B_{i}^{\ast})$ with respect to $w^{q}(x) dx$ and since $w(A_j x) \lesssim w(x)$ a.e. $x \in \mathbb{R}^{n}$, $1 \leq j \leq m$, we obtain that
\begin{equation} \label{Talfa 3}
J_2 \lesssim (w^{p}(B))^{-q/p} \int_{\mathbb{R}^{n}} \left[ M_{\frac{\alpha n}{n+d+1}}(\chi_{B}) (x) \right]^{q\frac{n+d+1}{n}} w^{q}(x) dx. 
\end{equation}
Being  $d= \lfloor n(\frac{1}{p}-1) \rfloor$, we have $q \frac{n+d+1}{n} > p \frac{n+d+1}{n} > 1$. We write $\widetilde{q} = q \frac{n+d+1}{n}$ and let $\frac{1}{\widetilde{p}} = \frac{1}{\widetilde{q}} + \frac{\alpha}{n+d+1}$, so $\frac{\widetilde{p}}{\widetilde{q}} = \frac{p}{q}$ and
$w^{q/{\widetilde{q}}} \in \mathcal{A}_{\widetilde{p}, \widetilde{q}}$ (see Remark \ref{w_1_q}). From Theorem \ref{Teo A_pq} we obtain
\[
\int_{\mathbb{R}^{n}} \left[ M_{\frac{\alpha n}{n+d+1}}(\chi_{B}) (x) \right]^{q\frac{n+d+1}{n}} w^{q}(x) dx \lesssim \left( \int_{\mathbb{R}^{n}} \chi_{B}(x) w^{p}(x) dx \right)^{q/p} =  (w^{p}(B))^{q/p}.
\]
This inequality and (\ref{Talfa 3}) give
\begin{equation}
J_2 \leq C. \label{Talfa 4}
\end{equation}
Finally, (\ref{Talfa 2}) and (\ref{Talfa 4}) allow us to obtain (\ref{uniform estimate}). This completes the proof.
\end{proof}

In the following corollary we recover Theorem 4.10 obtained in \cite{Rocha}.

\begin{corollary} For $0 < \alpha < n$, let $I_{\alpha}$ be the Riesz potential on $\mathbb{R}^{n}$. If $w^{\frac{n}{(n-\alpha)s}} \in \mathcal{A}_{1}$ 
with $0 < s < 1$, and $\frac{r_w}{r_w - 1} < \frac{n}{\alpha}$, then $I_{\alpha}$ can be extended to an 
$H^{p}_{w^{p}}(\mathbb{R}^{n})$-$L^{q}_{w^{q}}(\mathbb{R}^{n})$ bounded operator for each $s \leq p \leq 1$ and 
$\frac{1}{q} = \frac{1}{p} - \frac{\alpha}{n}$.
\end{corollary}

\begin{proof} To apply Theorem \ref{Ta} with $A_j = id$ for all $j= 1, ..., m$.
\end{proof}

\

\end{document}